\documentclass[a4paper,10pt]{article}

\usepackage[utf8]{inputenc}
\usepackage[T1]{fontenc}

\usepackage[english]{babel}

\usepackage{enumitem} 
\newlist{steps}{enumerate}{1}
\setlist[steps, 1]{wide=0pt, label=Step \arabic*., font=\scshape}

\usepackage{authblk} 

\usepackage{cancel}

\usepackage{geometry}

\usepackage{amsmath,amssymb,amsfonts,amsthm,amsbsy,stmaryrd,mathrsfs}
\usepackage{bbold} 
\usepackage{nccmath} 
\usepackage{units}   

\usepackage{graphicx}
\usepackage{xcolor}

\usepackage{hyperref}

\usepackage[capitalise]{cleveref} 

\newcommand{\eps}{\varepsilon}
\DeclareMathOperator{\sgn}{sgn}
\newcommand{\dd}{\mathop{}\!\mathrm{d}}
\newcommand{\abs}[1]{\left\lvert#1\right\rvert}

\renewcommand{\textbf}[1]{\begingroup\bfseries\mathversion{bold}#1\endgroup} 
\newcommand{\citer}[2][]{#1 in \cite{#2}} 

\theoremstyle {definition} 
\newtheorem{ex}{Example}[section]
\theoremstyle{plain} \newtheorem {theo} {Theorem}[section]

\newtheorem{prop}[theo]{Proposition}
\newtheorem{lem}[theo]{Lemma}
\newtheorem{remark}[theo]{Remark}

\theoremstyle{plain}

\theoremstyle{definition}

\begin{document}

\title{Asymptotic behaviour for a time-inhomogeneous 
Kolmogorov type diffusion}
\author{Mihai Gradinaru and Emeline Luirard}

\affil{\small Univ Rennes, CNRS, IRMAR - UMR 6625, F-35000 Rennes, France\\
\texttt{\small \{Mihai.Gradinaru,Emeline.Luirard\}@univ-rennes1.fr}}
\date{}
\maketitle

{\small\noindent {\bf Abstract:}~ We study a kinetic stochastic model with a non-linear time-inhomogeneous drag force and a Brownian-type random force. More precisely, the Kolmogorov type diffusion $(V,X)$ is considered: here $X$ is the position of the particle and $V$ is its velocity and is solution of a stochastic differential equation driven by a one-dimensional Brownian motion, with the drift of the form $t^{-\beta}F(v)$. The function $F$ satisfies some homogeneity condition and $\beta$ is positive. The behaviour of the process $(V,X)$ in large time is proved by using stochastic analysis tools.}\\
{\small\noindent{\bf Keywords:}~ kinetic stochastic equation; time-inhomogeneous diffusions; explosion times; scaling transformations; asymptotic distributions; ergodicity.}\\
{\small\noindent{\bf MSC2010 Subject Classification:}~Primary~60J60; Secondary~60H10; 60J65; 60F17.}

\section{Introduction}

In several domains as fluids dynamics, 
statistical mechanics, biology, a number of models are based on the Fokker-Planck and Langevin equations driven by Brownian motion or could be non-linear or driven by other random noises.
For example, in \cite{CattiauxAsymptoticanalysisdiffusion2010} the persistent turning walker model was introduced, inspired by the modelling of fish motion. An associated two-component  Kolmogorov  type diffusion solves
a kinetic Fokker-Planck equation based on an Ornstein-Uhlenbeck Gaussian process and the authors studied 
the large time behaviour of this model by using appropriate tools from stochastic analysis. 
One of the natural questions is the behaviour in large time of the solution to the corresponding stochastic differential equation (SDE).
 Although the tools of partial differential equations allow us to ask of this kind of questions, since these models are probabilistic, tools based on stochastic processes could be more natural to use.
\medskip

In the last decade 
the asymptotic study of solutions of non-linear Langevin's type was the subject of an important number of papers,
see \cite{CattiauxDiffusionlimitkinetic2019}, \cite{EonGaussianasymptoticsnonlinear2015}, \cite{FournierOnedimensionalcritical2021}. For instance, in \cite{FournierOnedimensionalcritical2021}  
the following system is studied 
\[V_t=v_0+B_t-\frac{\rho}{2}\int_{0}^{t}F(V_s)\dd s
\quad\mbox{ and }\quad X_t=x_0+\int_{0}^{t}V_s \dd s.
\]
In other words one considers a particle moving such that its velocity is a diffusion with an invariant measure behaving like $(1+|v|^2)^{-\rho/2}$, as $|v|\to+\infty$. 
The authors prove that for large time, after a suitable rescaling, the position process behaves as a Brownian motion or other stable processes, following the values of $\rho$. Results have been extended to additive functional of $V$ in \cite{BethencourtStablelimittheorems2021}. It should be noticed that these cited papers use the standard tools associated with time-homogeneous equations: invariant measure, scale function and speed measure. Several of these tools will not be available when the drag force is depending explicitly on time. 
In \cite{GradinaruExistenceasymptoticbehaviour2013}, a non-linear SDE driven by a Brownian motion but having time-inhomogeneous drift coefficient was studied and its large time behaviour was described. Moreover, sharp rates of convergence are proved for the 1-dimensional marginal of the solution. In the present paper, we consider the velocity process as satisfying the same kind of SDE. 
\bigskip

Let us describe our framework: consider a one-dimensional time-inhomogeneous  stochastic kinetic model driven by a Brownian motion. We  denote by $(X_t)_{t\geq0}$ the one-dimensional process describing the position of a particle at time $t$ having the velocity $V_t$. The velocity process $(V_t)_{t\geq 0}$ is supposed to follow a Brownian dynamic in a  potential $U(t,v)$, varying in time:
\begin{equation*}\label{position}
\dd V_t=\dd B_t-\dfrac{1}{2}\partial_vU(t,V_t)\dd t\quad
\mbox{ and }\quad
X_t=X_0+\int_{0}^{t}V_s\dd s.
\end{equation*}
This system can be viewed as a perturbation of the classical two-component Kolmogorov diffusion 
\[
\dd V_t=\dd B_t\quad
\mbox{ and }\quad
X_t=X_0+\int_{0}^{t}V_s\dd s.
\]
In the present paper the potential is supposed to grow slowly to infinity, and it will be supposed to be of the form $t^{-\beta}\int_{0}^{v}F(u)\dd u$, with $\beta>0$ and $F$ satisfying some homogeneity condition. It describes a one dimensional particle evolving in a force field  $Ft^{-\beta}$ and undergoing many small random shocks.
A natural question is to understand the behaviour of the process $(V,X)$ in large time. More precisely we look for the limit in distribution of 
$v(\eps)(V_{t/\eps},\eps X_{t/\eps})_t$, as $\eps \to 0$, where $v(\eps)$ is some rate of convergence.
Our results are proved on the product of path spaces and consequently contain those of \cite{GradinaruExistenceasymptoticbehaviour2013}. \\
If $F=0$, it is not difficult to see that the rescaled position process 
$(\eps^{\nicefrac{1}{2}}V_{t/\eps}, \eps^{\nicefrac{3}{2}}X_{t/\eps})_t$ converges in distribution towards the Kolmogorov diffusion $(B_t, \int_0^t B_s \dd s)_t$. 
We prove that this kinetic behaviour still holds for sufficiently "small at infinity" potential.
The strategy to tackle this problem is based on estimates of moments of the velocity process.
The main result can then be extended for the case when the potential is equally weighted in some sense as the random noise. The potential either offsets the random noise (critical regime) or swings with it (sub-critical regime).
\\
As suggested at the beginning of the introduction, other random noises can be considered. In \cite{GradinaruAsymptoticbehaviortimeinhomogeneous2021}, the case of a Lévy random noise is analysed. The case of a stochastic system in a harmonic potential is the purpose of a future work (see \cite{LuirardLoilimitemodeles2022}).

\bigskip
The organisation of our paper is as follows: in the next section we introduce notations, and we state our main results. Results about existence and non-explosion of solutions are stated in Section 3. Estimates of the moments of the velocity process are given in Section 4 while the proofs of our main results are presented in Section 5. 

\section{Notations and main results}
Let $(B_t)_{t\geq0}$ be a standard Brownian motion, $\beta$ a real number and $F$ a continuous function which is supposed to satisfy either 
\begin{equation}\tag{$H_1^{\gamma}$}\label{hyp1_levy}
	\mbox{for some }\gamma \in \mathbb{R},\ \forall v \in \mathbb{R},\ \lambda>0, \ F(\lambda v)=\lambda^{\gamma}F(v),
\end{equation}
or
\begin{equation}\tag{$H_2^{\gamma}$}\label{hyp2_levy}
	\abs{F}\leq G\mbox{ where }G \mbox{ is a positive function satisfying \eqref{hyp1_levy}.}
\end{equation}

Each assumption implies that there exist a positive constant $K$ such that, for all $v\in \mathbb{R}$, $\abs{F(v)}\leq K\abs{v}^{\gamma}$.
Obviously \eqref{hyp2_levy} is a generalisation of \eqref{hyp1_levy}. 
In the following, $\sgn$ is the sign function with convention $\sgn(0)=0$.
As an example of function satisfying \eqref{hyp1_levy}  one can keep in mind $F:v\mapsto \sgn(v)\abs{v}^{\gamma}$ (see also \cite{GradinaruExistenceasymptoticbehaviour2013}), and as an example of function satisfying \eqref{hyp2_levy} 
(with $\gamma=0$) 
$F:v\mapsto \nicefrac{v}{(1+v^2)}$ (see also \cite{FournierOnedimensionalcritical2021}). 

\begin{remark}\label{pi}
	If a function $\pi$ satisfies \eqref{hyp1_levy}, then for all $x\in \mathbb{R}$,  $\pi(x)=\pi(\sgn(x))\abs{x}^{\gamma}$.
\end{remark}
We consider the following one-dimensional stochastic kinetic model, for $t\geq t_0>0$,
\begin{equation}\label{equation time_levy}\tag{\text{SKE}}
	\dd V_t= \dd B_t-t^{-\beta}F(V_{t})\dd t,\ V_{t_0}=v_0>0, 
	\quad\mbox{ and }\quad 
	\dd X_t=V_t \dd t, \ X_{t_0}=x_0\in \mathbb{R}.
\end{equation} 
Most of the convergences take place in the space of continuous functions ${\mathcal{C}}((0,+\infty), \mathbb{R})$ endowed by the uniform topology
\[\displaystyle \dd_u:f,g\in \mathcal{C}((0,+\infty),\mathbb{R})\mapsto \sum_{n=1}^{+\infty}\dfrac{1}{2^n}\min\Big(1,\sup_{[\frac{1}{n},n]}\abs{f-g}\Big).\]
For a family $((Z_t^{(\eps)})_{t> 0})_{\eps>0}$ of continuous processes, we write
\[(Z_t^{(\eps)})_{t> 0} \quad \Longrightarrow \quad (Z_t)_{t>0}, \]
if $(Z_t^{(\eps)})_{t>0}$ converges in distribution to $(Z_t)_{t>0}$ in ${\mathcal{C}}((0,+\infty), \mathbb{R})$, as $\eps \to 0$.\\
We write \[(Z_t^{(\eps)})_{t> 0} \quad \stackrel{f.d.d.}{\Longrightarrow}\quad  (Z_t)_{t>0}, \]
if for all finite subsets $S \subset (0,+\infty)$, the vector $(Z_t^{(\eps)})_{t\in S}$ converges in distribution to $(Z_t)_{t\in S}$ in $\mathbb{R}^S$, as $\eps \to 0$.\\
Let us state our main results. Set $q:=\dfrac{\beta}{\gamma+1}$.
\begin{theo}\label{main_thm_inh} 
	Consider $\gamma\geq 0$, and $q >\frac{1}{2}$. Assume that either \eqref{hyp1_levy} or \eqref{hyp2_levy} is satisfied. Let $(V_t,X_t)_{t\geq t_0}$ be the solution to \eqref{equation time_levy} and $(\mathcal{B}_t)_{t\geq0}$ be a standard Brownian motion. Furthermore, if $\gamma\geq 1 $, we suppose that for all $v\in \mathbb{R}$,
	$vF(v)\geq 0$.
	\\Then, as $\eps\to 0$, 
	\begin{equation*}
	\left(\sqrt{\eps}V_{t/\eps}, \eps^{\nicefrac{3}{2}}X_{t/\eps}\right)_{t\geq \eps t_0}\stackrel{\text{}}{\Longrightarrow}\left(\mathcal{B}_{t}, \int_{0}^t\mathcal{B}_s \dd s\right)_{t\geq0}.
	\end{equation*}
\end{theo}
\begin{theo}\label{critique}
	Consider $\gamma\geq 0$ and $q=\frac{1}{2}$. Assume that \eqref{hyp1_levy} is satisfied. Let $(V_t,X_t)_{t\geq t_0}$ be the solution to \eqref{equation time_levy}. If $\gamma\geq 1$, we suppose furthermore that for all $v\in \mathbb{R}$, 
	$vF(v)\geq 0$.
	\\
	Call $\widetilde{H}$ the eternal ergodic process, solution to the homogeneous SDE 
	\begin{equation*}
	\dd H_s=\dd W_s-\dfrac{H_s}{2}\dd s-F\big(H_s\big)\dd s,
	\end{equation*} such that the law of $H_{-\infty}$ is the invariant measure, where $(W_t)_{t\geq0}$ is again a standard Brownian motion. Setting $\Lambda_{F, t_1, \cdots, t_d}$ for the f.d.d. of $\widetilde{H}$, we call $(\mathcal{V}_t)_{t\geq0}$ the process whose finite dimensional distribution (f.d.d.) are $T*\Lambda_{F,\log(t_1), \cdots, \log(t_d)}$, the pushforward measure of $\Lambda_{F,\log(t_1), \cdots, \log(t_d)}$ by the linear map $T(u_1, \cdots, u_d):=(\sqrt{t_1}u_1, \cdots, \sqrt{t_d}u_d)$, that is $(\mathcal{V}_t)_{t \geq0}=(\sqrt{t}\widetilde{H}_{\log(t)})_{t \geq0}$. \\Then, as $\eps\to 0$, 
	\begin{equation*}
	\left(\sqrt{\eps}V_{t/\eps}, \eps^{\nicefrac{3}{2}}X_{t/\eps}\right)_{t\geq \eps t_0}\stackrel{\text{}}{\Longrightarrow}\left(\mathcal{V}_{t}, \int_{0}^t\mathcal{V}_s \dd s\right)_{t\geq0}.
	\end{equation*}
\end{theo}
\begin{remark}
		The one-dimensional distribution of $(\mathcal{V}_t)_{t\geq0}$ has already been explicitly computed (see \citer[Theorem 4.1]{GradinaruExistenceasymptoticbehaviour2013}).
\end{remark}
\begin{theo}\label{sub-critical}
	Consider $\gamma\geq 1$ and $q<\frac{1}{2}$. Assume that $F:v\mapsto \rho \sgn(v)\abs{v}^{\gamma}$ with $\rho>0$. Let $(V_t,X_t)_{t\geq t_0}$ be the solution to \eqref{equation time_levy}. 
	Call $\widehat{H}$ the ergodic process, solution to the homogeneous SDE
	\begin{equation*}
		\dd H_s=\dd W_s-F\left(H_s\right)\dd s,
	\end{equation*} where $(W_t)_{t\geq0}$ is a standard Brownian motion. Call $\Pi_F$ its invariant measure. We call $(\mathscr{V}_t)_{t\geq0}$ the process whose f.d.d.\ are $T*\left(\otimes^d \Pi_{F}\right)$, the pushforward measure of $\otimes^d \Pi_{F}$ by the linear map $T(u_1, \cdots, u_d):=({t_1}^qu_1, \cdots, {t_d}^qu_d)$. \\
	Then, as $\eps \to 0$,
	\begin{equation*}
		\left(\eps^{q}V_{t/\eps}\right)_{t\geq \eps t_0}\stackrel{\text{f.d.d.}}{\Longrightarrow}\left(\mathscr{V}_{t}\right)_{t\geq0}.
		\end{equation*}
	Moreover, in the linear case (i.e.\ $\gamma=1$) and if $\beta>-\frac{1}{2}$, we define $(\mathscr{X}_t)_{t\geq0}$ the centered Gaussian process with covariance function $K(s,t):= (\rho^2(1+2\beta))^{-1}(s\wedge t)^{1+2\beta}$. \\	Then, as $\eps \to 0$,
	\begin{equation}\label{eq: cv_X_linear}
		\left(\eps^{\beta+\frac{1}{2}}X_{t/\eps}\right)_{t\geq \eps t_0}\stackrel{\text{f.d.d.}}{\Longrightarrow}\left(\mathscr{X}_{t}\right)_{t\geq0}.
		\end{equation}
\end{theo}
\begin{remark}
	If $\beta=0$, one can prove using the martingale method, that $(\sqrt{\eps}X_{t/\eps})_{t\geq 0}$ converges towards a Brownian motion. Assume, by way of contradiction, that the process $(\eps^{q}V_{t/\eps})_{t\geq \eps t_0}$ would converge (i.e. were tight), then by the continuous mapping theorem, the process $(\eps X_{t/\eps})_{t\geq 0}$ should converge. This is a contradiction with \eqref{eq: cv_X_linear}. Here is why we deal only with finite-dimensional convergence for the velocity process.
\end{remark}

\section{Changed-of-time processes}
In the following, we suppose that $\gamma>-1$ and 
set $\Omega=\overline{\mathcal{C}}([t_0,+\infty))$ the set of continuous functions, that equal $+\infty$ after their (possibly infinite) explosion time. 
Following the idea used in \cite{GradinaruExistenceasymptoticbehaviour2013}, we first perform a change of time in \eqref{equation time_levy} in order to produce at least one time-homogeneous coefficient in the transformed equation.
For every ${\mathcal{C}}^2$-diffeomorphism $\varphi: [0,t_1)\to [t_0,+\infty)$, let introduce the scaling transformation $\Phi_{\varphi}$ defined, for $\omega \in \Omega $, by
\[\Phi_{\varphi}(\omega)(s):=\dfrac{\omega(\varphi(s))}{\sqrt{\varphi'(s)}}\text{, with }s\in [0,t_1).\]

\noindent
The result containing the change of time transformation can be found in \cite{GradinaruExistenceasymptoticbehaviour2013}, Proposition 2.1, p. 187.\\
Let $V$ be solution to the equation \eqref{equation time_levy}. Thanks to Lévy's characterization theorem of the Brownian motion, $\left(W_t \right)_{t\geq0}:=\left(\medint\int_{0}^{t}\dfrac{\dd B_{\varphi(s)}}{\sqrt{\varphi'(s)}}\right)_{t\geq 0}$ is a standard Brownian motion. Then, by a change of variable $t=\varphi(s)$, one gets
\[V_{\varphi(t)}-V_{\varphi(0)}=\int_{0}^{t}\sqrt{\varphi'(s)}\dd W_s-\int_{0}^{t}\dfrac{F(V_{\varphi(s)})}{\varphi(s)^{\beta}}\varphi'(s)\dd s.\]
The integration by parts formula yields
\[\dd \left(\dfrac{V_{\varphi(s)}}{\sqrt{\varphi'(s)}}\right)=\dd W_s-\dfrac{\sqrt{\varphi'(s)}}{\varphi(s)^{\beta}}F(V_{\varphi(s)})\dd s -\dfrac{\varphi''(s)}{2\varphi'(s)}\dfrac{V_{\varphi(s)}}{\sqrt{\varphi'(s)}}\dd s.\] 
As a consequence, we can state the following result in our context.
\begin{prop}\label{changeoftime}
	If $V$ is a solution to the equation \eqref{equation time_levy}, then $V^{(\varphi)}:=\Phi_{\varphi}(V)$ is a solution to 
	\begin{equation}\label{change of time equation}
	\dd V^{(\varphi)}_s=\dd W_s -\dfrac{\sqrt{\varphi'(s)}}{\varphi(s)^{\beta}}F(\sqrt{\varphi'(s)}V_s^{(\varphi)})\dd s -\dfrac{\varphi''(s)}{\varphi'(s)}\dfrac{V_s^{(\varphi)}}{2}\dd s, \  V_{0}^{(\varphi)}=\dfrac{V_{\varphi(0)}}{\sqrt{\varphi'(0)}},
	\end{equation} where $W_t:=\medint\int_{0}^{t}\dfrac{\dd B_{\varphi(s)}}{\sqrt{\varphi'(s)}}$.
	
	\noindent		
	If $V^{(\varphi)}$ is a solution to \eqref{change of time equation}, then $\Phi_{\varphi}^{-1}(V^{(\varphi)})$ is a solution to  the equation \eqref{equation time_levy}, where $B_t-B_{t_0}:=\medint\int_{t_0}^{t}\sqrt{(\varphi'\circ \varphi^{-1})(s)}\dd W_{\varphi^{-1}(s)}$.
	
	\noindent
	Furthermore, uniqueness in law, pathwise uniqueness or strong existence hold for the equation \eqref{equation time_levy} if and only if they hold for the equation \eqref{change of time equation}.
\end{prop}
In the following, we will use two particular changes of time, depending on which term of \eqref{change of time equation} should become time-homogeneous.
\begin{itemize}
	\item {\sl The exponential change of time}: denoting $\varphi_e: t \mapsto t_0 e^t$, the exponential scaling transformation is defined by $\Phi_e(\omega):s\in \mathbb{R}^{+}\mapsto \dfrac{\omega_{t_0 e^s}}{\sqrt{t_0}e^{\nicefrac{s}{2}}}$, for $\omega\in \Omega $. 
	Thanks to \cref{changeoftime}, the process $V^{(e)}:=\Phi_e(V)$ satisfies the equation 
	\begin{equation*}\label{equation Ve}
	\dd V_s^{(e)}=\dd W_s-\dfrac{V_s^{(e)}}{2}\dd s- t_0^{\nicefrac{1}{2}-\beta} e^{(\nicefrac{1}{2}-\beta)s}F\big(\sqrt{t_0}e^{\nicefrac{s}{2}}V_s^{(e)}\big)\dd s,
	\end{equation*} where $(W_t)_{t\geq0}$ is a standard Brownian motion.
	\item {\sl The power change of time}: for $q= \frac{\beta}{\gamma+1}\neq \frac{1}{2}$, consider $\varphi_{q}\in \mathcal{C}^2([0,t_1))$ the solution to the Cauchy problem
	\[\varphi_{q}'=\varphi_{q}^{2q}, \  \varphi_{q}(0)=t_0. \]
	Clearly, $\varphi_{q}(t)=\big(t_0^{1-2q}+(1-2q)t\big)^{\nicefrac{1}{(1-2q)}}$, when $2q\neq1$, and $\varphi_{q}=\varphi_e$, when $2q=1$.

	The time $t_1$ satisfies $t_1=+\infty$, when $2q\leq1$, and  $t_1=t_{0}^{1-2q}(2q-1)^{-1}$, when $2q>1$. The power scaling transformation is defined by $\Phi_q(\omega):s\in \mathbb{R}^{+}\mapsto \dfrac{\omega(\varphi_{q}(s))}{\varphi_{q}(s)^{q}}$. 
	The process $V^{(q)}:=V^{(\varphi_{q})}$ satisfies the equation
	\begin{equation}\label{power hyp2}
	\dd V_s^{(q)} =\dd W_s- \varphi_{q}^{-\gamma q}(s) F\Big(\sqrt{\varphi'_{q}(s)}V_s^{(q)}\Big)\dd s-q \varphi_{q}^{2q-1}(s)V_s^{(q)}\dd s,
	\end{equation}
	where $(W_t)_{t\geq0}$ is a standard Brownian motion.
\end{itemize}
Adapting the proof of \citer[Propositions 3.2, 3.6 and 3.7 p. 188,]{GradinaruExistenceasymptoticbehaviour2013}, one can prove the following proposition.
\begin{prop}\label{existence hyp1} \label{explosion time}
	For $\gamma\geq0 $, there exists a pathwise unique strong solution to \eqref{equation time_levy}, defined up to the explosion time $\tau_{\infty}$ of $V$. 
	\begin{itemize}
		\item When $\gamma\leq 1$ or for all $v\in \mathbb{R}$,
			$vF(v)\geq 0$, then
		  $\tau_{\infty}$ is a.s.\ infinite. 
		\item When $2q>1$, then $\mathbb{P}(\tau_{\infty}=+\infty)>0$.	
		\item Under \eqref{hyp1_levy}, if $\gamma>1$ and $(F(-1),F(1))\in ((0,+\infty))\times [0,+\infty))\cup (\mathbb{R}\times (-\infty,0))$, then $\mathbb{P}(\tau_{\infty}=+\infty)<1$. 
			
	\end{itemize} 
\end{prop}
\begin{remark}
	Assume that \eqref{hyp1_levy} is satisfied. In the linear case ($\gamma=1$), the drift and the diffusion terms are Lipschitz and satisfy locally linear growth condition. The existence and non-explosion of $V$ follow from \citer[Theorem 2.9, p. 289,]{KaratzasBrownianMotionStochastic1998}.
\end{remark}

For more details, we refer to \cite{LuirardLoilimitemodeles2022}.

\section{Moment estimates of the velocity process}

In this section, we give estimates for the moment of the velocity process. It will be useful to control some stochastic terms appearing later.
\begin{prop}\label{esperance}~ 
	Assume that $\gamma\geq 0$ and $\beta \in \mathbb{R}$. 
	The inequality \[\forall t\geq  t_0, \ \mathbb{E}\left[ \abs{V_t}^{\kappa} \right]\leq C_{\gamma,\kappa,\beta,t_0}t^{\frac{\kappa}{2}}\] holds for
	\begin{itemize}
		\item $\kappa\in [0,1]$, when $\gamma< 1$ and $\beta\geq \frac{\gamma+1}{2}$, 
		\item $\kappa\geq 0$, when for all $v\in \mathbb{R}$, 
		$vF(v)\geq0$.
	\end{itemize} 
	If $\kappa\in [0,1]$, $\gamma< 1$ and $\beta< \frac{\gamma+1}{2}$, then
	\[\forall t\geq  t_0, \ \mathbb{E}\left[ \abs{V_t}^{\kappa} \right]\leq C_{\gamma,\kappa,\beta,t_0}t^{\kappa \frac{1-\beta}{1-\gamma}}.\]
\end{prop}
	\begin{remark}
	When $-1<\gamma<0$, it can be proved that for all $t\geq t_0$, $\mathbb{E}\left[ \abs{V_t} \right]\leq C_{\gamma,\beta,t_0}\sqrt{t}$, without hypothesis of the positivity of the function $v\mapsto vF(v)$.
\end{remark}
\begin{proof}
	\begin{steps}
	\item Assume that $\gamma\geq1$ and that for all $v\in \mathbb{R}$, 
	$vF(v)\geq0$.\\
	 Define, for all $n\geq 0$, the stopping times $T_n:= \inf\{t\geq t_0, \ \abs{V_t}\geq n \}$.
	By Itô's formula, for all $t\geq t_0$, we have
	\[\begin{aligned}
	V_{t\wedge T_n}^2 & = v_0^2+ \int_{t_0}^{t\wedge T_n} 2V_s \dd B_s - \int_{t_0}^{t\wedge T_n} 2s^{-\beta}V_sF(V_s)\dd s+({t\wedge T_n}-t_0) \\&=
	v_0^2+ \int_{t_0}^{t}\mathbb{1}_{s\leq T_n} 2V_s \dd B_s 
	- \int_{t_0}^{t\wedge T_n} 2s^{-\beta}V_sF(V_s)\dd s
	+(t\wedge T_n-t_0)
	\\ &\leq  v_0^2+ \int_{t_0}^{t}\mathbb{1}_{s\leq T_n} 2V_s \dd B_s+({t}-t_0).
	\end{aligned}\]
	Since $\int_{t_0}^{t}4\mathbb{1}_{s\leq T_n} V_s^2 \dd s \leq 4n^2(t-t_0)  <+\infty$, taking expectation yields 
	\[ \mathbb{E}\left[ V_{t\wedge T_n}^2 \right]\leq v_0^2+(t-t_0) \leq C_{t_0}t.\]  
	Set $\kappa\in [0,2] $, we obtain by Jensen's inequality that
	\begin{equation}\label{estimates}
	\mathbb{E}\left[ \abs{V_t}^{\kappa}\right]\leq \mathbb{E}\left[ \abs{V_t}^{2}\right] ^{\frac{\kappa}{2}} \leq \left( \liminf_{n\to +\infty}\mathbb{E}\left[ V_{t\wedge T_n}^2 \right] \right)^{\frac{\kappa}{2}} \leq  C_{\kappa, t_0}t^{\frac{\kappa}{2}}.
	\end{equation}  
	When $\kappa > 2$, the function $v\mapsto \abs{v}^{\kappa}$ is $\mathcal{C}^2$, so by Itô's formula, we can write for all $t\geq t_0$,
	\begin{multline*}
	\abs{V_{t\wedge T_n}}^{\kappa} = \abs{v_0}^{\kappa}+ \int_{t_0}^{t\wedge T_n} \kappa \sgn(V_s) \abs{V_s}^{\kappa-1} \dd B_s - \int_{t_0}^{t\wedge T_n} \kappa s^{-\beta}\abs{V_s}^{\kappa-1}\sgn(V_s)F(V_s)
	\dd s\\+\int_{t_0}^{t\wedge T_n}\dfrac{\kappa(\kappa-1)}{2}\abs{V_s}^{\kappa-2}\dd s.
	\end{multline*}
	In addition, using the hypothesis on the sign of $F$, we have
	\begin{equation}\label{Ito}
		\abs{V_{t\wedge T_n}}^{\kappa} \leq\abs{v_0}^{\kappa} + \int_{t_0}^{t}\mathbb{1}_{s\leq T_n} \kappa \sgn(V_s) \abs{V_s}^{\kappa-1} \dd B_s+\int_{t_0}^{t\wedge T_n}\dfrac{\kappa(\kappa-1)}{2}\abs{V_s}^{\kappa-2}\dd s.
	\end{equation}
We observe that $\int_{t_0}^{t}\kappa^2V_s^{2\kappa-2}\mathbb{1}_{s\leq T_n} \dd s \leq \kappa^2n^{2\kappa-2}(t-t_0)<+\infty$. Taking expectation in \eqref{Ito}, we obtain
\[	\mathbb{E}\left[\abs{V_{t}}^{\kappa} \right] \leq \liminf_{n\to +\infty} \mathbb{E}\left[\abs{V_{t\wedge T_n}}^{\kappa}   \right] \leq \abs{v_0}^{\kappa}+  \int_{t_0}^{t}\dfrac{\kappa(\kappa-1)}{2}\mathbb{E}\left[\abs{V_s}^{\kappa-2}\right]\dd s.
\] 
When $0\leq \kappa -2 \leq 2$, we can upper bound $\mathbb{E}\left[\abs{V_s}^{\kappa-2}\right]$ by injecting \eqref{estimates} and get
\[\mathbb{E}\left[\abs{V_{t}}^{\kappa} \right]  \leq  \abs{v_0}^{\kappa}+  \int_{t_0}^{t}\dfrac{\kappa(\kappa-1)}{2}C_{\kappa, t_0}s^{\frac{\kappa-2}{2}}\dd s \leq C_{\kappa, t_0}s^{\frac{\kappa}{2}}.\]
The same method is then applied inductively to prove the inequality for all $\kappa>2$.
	\item Assume now that $\gamma\in [0,1[$. Fix $\kappa \in [0,1]$. Then Jensen's inequality yields, for all $t\geq t_0$, $\mathbb{E}\left[\abs{V_{t}}^\kappa \right]\leq \mathbb{E}\left[\abs{V_{t}}\right]^\kappa$, hence it suffices to verify the inequality only for $\kappa=1$.\\
	Define, for all $n\geq 0$, the stopping times $T_n:= \inf\{t\geq t_0, \ \abs{V_t}\geq n \}$ and let us recall that under both hypotheses \eqref{hyp1_levy} or \eqref{hyp2_levy}, there exists a positive constant $K$, such that $\abs{F\left(v\right)}\leq K\abs{v}^{\gamma}$. We can write, for $t\geq t_0$ and $n\geq0$,
	\[\begin{aligned}
	\abs{V_{t\wedge T_n}}& \leq \abs{v_0-B_{t_0}}+\abs{B_{t\wedge T_n}}+\int_{t_0}^{t\wedge T_n}s^{-\beta}\abs{F(V_{s\wedge T_n}) }\dd s\\ &\leq \abs{v_0-B_{t_0}}+\abs{B_{t\wedge T_n}}+\int_{t_0}^{t\wedge T_n}s^{-\beta}K\abs{V_{s\wedge T_n} }^{\gamma}\dd s.
	\end{aligned}\]
	By noting that $\gamma\in [0,1[$ and $(B^2_t-t)_{t\geq 0}$ is a martingale, taking expectation we get
	\[\begin{aligned}
	\mathbb{E}\left[ \abs{V_{t\wedge T_n}}  \right] & \leq \mathbb{E}\left[\abs{v_0-B_{t_0}} \right] +\mathbb{E}\left[\abs{B_{t\wedge T_n}} \right] +\int_{t_0}^{t}s^{-\beta}K\mathbb{E}\left[\abs{V_{s\wedge T_n} }^{\gamma}\right]\dd s \\ &\leq \mathbb{E}\left[\abs{v_0-B_{t_0}} \right] +\sqrt{\mathbb{E}\left[B_{t\wedge T_n}^2 \right]} +\int_{t_0}^{t}s^{-\beta}K\mathbb{E}\left[\abs{V_{s\wedge T_n} }\right]^{\gamma}\dd s \\ & \leq \mathbb{E}\left[\abs{v_0-B_{t_0}} \right] +\sqrt{\mathbb{E}\left[{t\wedge T_n} \right]} +\int_{t_0}^{t}s^{-\beta}K\mathbb{E}\left[\abs{V_{s\wedge T_n} }\right]^{\gamma}\dd s\\ & \leq C_{t_0}\sqrt{t} +\int_{t_0}^{t}s^{-\beta}K\mathbb{E}\left[\abs{V_{s\wedge T_n} }\right]^{\gamma}\dd s.
	\end{aligned} \] The function $g_n:t\mapsto\mathbb{E}\left[ \abs{V_{t\wedge T_n}}  \right]$ is bounded by $n$. Applying a Gronwall-type lemma, stated below (\cref{gronwall}) 
	and Fatou's lemma, for $\beta\neq 1$ and for all $t \geq t_0$, we end up with
	\[\begin{aligned} \ \mathbb{E}\left[ \abs{V_{t}}  \right]\leq \liminf_{n\to +\infty}\mathbb{E}\left[ \abs{V_{t\wedge T_n}}  \right]&\leq C_{\gamma} \left[C_{t_0}\sqrt{t}+\left(\dfrac{1-\gamma}{1-\beta}K(t^{1-\beta}-t_0^{1-\beta})\right)^{\frac{1}{1-\gamma}}  \right]\\ & \leq
		C_{\gamma, \beta, t_0}\begin{cases} \sqrt{t} & \mbox{ if }\beta \geq \frac{\gamma+1}{2},\\
			t^{\frac{1-\beta}{1-\gamma}} & \mbox{else.}
		\end{cases}
	\end{aligned}\]
	The case $\beta=1$ can be treated similarly.
\end{steps}
\end{proof}

\begin{lem}[Gronwall-type lemma] \label{gronwall}
	Fix $r\in [0,1)$ and $t_0\in \mathbb{R}$.
	Assume that $g$ is a non-negative real-valued function, $b$ is a positive function and $a$ is a differentiable real-valued  function. Moreover, suppose that the function $bg^{r}$ is continuous. If
	\begin{equation}\label{eq:debut}
	\forall t \geq t_0, \ g(t)\leq a(t)+\int_{t_0}^{t}b(s)g(s)^{r}\dd s, 
	\end{equation} 
	then,
	\[\forall t \geq t_0, \ g(t)\leq 2^{\frac{1}{1-r}} \left[a(t)+\left((1-r)\int_{t_0}^tb(s)\dd s\right)^{\frac{1}{1-r}}  \right].\]
\end{lem}

\begin{proof}
	For $t\geq t_0$, since $r\geq 0$,
	\[ g(t)^{r}\leq \left(a(t)+\int_{t_0}^{t}b(s)g(s)^{r}\dd s\right)^{r}\,, \]
	then, multiplying by $b(t)>0$,
	\[ b(t)g(t)^{r}\leq b(t)\left(a(t)+\int_{t_0}^{t}b(s)g(s)^{r}\dd s\right)^{r}\,.\]
	Now, let us make appear the derivative of $H$
	\[a'(t)+ b(t)g(t)^{r}\leq a'(t)+ b(t)\left(a(t)+\int_{t_0}^{t}b(s)g(s)^{r}\dd s\right)^{r}, \]
	that is 
	\[\dfrac{a'(t)+ b(t)g(t)^{r}}{\left(a(t)+\int_{t_0}^{t}b(s)g(s)^{r}\dd s\right)^{r}}\leq b(t)+\dfrac{a'(t)}{\left(a(t)+\int_{t_0}^{t}b(s)g(s)^{r}\dd s \right)^{r}}\leq b(t)+\dfrac{a'(t)}{a(t)^{r}}. \]
	Integrating, since $r\neq1$, we obtain
	\[ (1-r)^{-1}\left[\left(a(t)+\int_{t_0}^{t}b(s)g(s)^{r}\dd s\right)^{1-r} -a(t_0)^{1-r}\right]\leq (1-r)^{-1}\left[a(t)^{1-r}-a(t_0)^{1-r} \right]+\int_{t_0}^tb(s)\dd s  \]
	or equivalently, setting $H$ for the right-hand side of \eqref{eq:debut} and using that $r<1$, we get
	\[ H(t)^{1-r}\leq a(t)^{1-r}+(1-r)\int_{t_0}^tb(s)\dd s . \]
	Since $\frac{1}{1-r}>0$ and using \eqref{eq:debut} 
	\[g(t)\leq \left( a(t)^{1-r}+(1-r)\int_{t_0}^tb(s)\dd s  \right)^{\frac{1}{1-r}}\leq C_{r} \left[a(t)+\left((1-r)\int_{t_0}^tb(s)\dd s\right)^{\frac{1}{1-r}}  \right].\]
	This concludes the proof of the lemma.
\end{proof}

\begin{remark}
	Call $H(t)$ the right-hand side of \eqref{eq:debut}.
	If $g$ is not continuous, note that the function $H$  
	is continuous and satisfies \eqref{eq:debut} (since $b$ is positive and $g\leq H$). Therefore, one can apply the lemma to $H$ and then use the inequality $g\leq H$. 
\end{remark}

\section{Proof of the asymptotic behaviour of the solution}
This section is devoted to the proofs of our main results.\\
\subsection{Asymptotic behaviour in the super-critical regime under both assumptions}

In this section, we assume that $\gamma\geq0$ and $q>\frac{1}{2}$.
\begin{proof}[Proof of \cref{main_thm_inh}] We split the proof into three steps.
	\begin{steps}
	\item We note that it is enough to prove that the process \[(V_t^{(\eps)})_{t\geq 0}:=(\sqrt{\eps}V_{t/\eps})_{t\geq 0}\] converges in distribution to a Brownian motion in the space of continuous functions ${\mathcal{C}}([0,+\infty))$ endowed by the uniform topology. In order to see $V^{(\eps)}$ as a process of $\mathcal{C}([0,+\infty))$, let us state for all $s\in [0,\eps t_0]$, $V_s^{(\eps)}:=V_{\eps t_0}^{(\eps)}=\sqrt{\eps}v_0 $.\\
	For every $\eps \in (0,1]$ and $t\geq \eps t_0$, 
we can write
	\[\eps^{\nicefrac{3}{2}}X_{t/\eps}=\eps^{\nicefrac{3}{2}} x_0 +\int_{\eps t_0}^{t}V_s^{(\eps)} \dd s. \] 
	Clearly, the theorem will be proved once we show that $g_{\eps}(V^{(\eps)}_{\bullet}):=(V^{(\eps)}_{\bullet},\int_{\eps t_0}^{\bullet}V_s^{(\eps)} \dd s)$ converges weakly  in ${\mathcal{C}}([0,+\infty))$ endowed by the uniform topology.
Here  the mapping $g_{\eps}: v\mapsto \left(v_t, \int_{\eps t_0}^{t}v_s \dd s \right)_{t\geq0}$ is defined and valued on ${\mathcal{C}}((0,+\infty))$. This mapping is converging, as $\eps\to 0$, to the continuous mapping $g:v\mapsto \left(v_t, \int_{0}^{t}v_s \dd s \right)_{t\geq0}$.\\
We have, for every $\eps \in (0,1]$ and $t\geq \eps t_0$,
	\[\begin{aligned}
		V_t^{(\eps)}=\sqrt{\eps}V_{t/\eps}=& \sqrt{\eps}(v_0 - B_{t_0})+ \sqrt{\eps}B_{t/\eps} - \sqrt{\eps}\int_{t_0}^{t/\eps}F(V_s)s^{-\beta}\dd s \\
		=&\sqrt{\eps}(v_0 - B_{t_0})+ B_{t}^{(\eps)} - \eps^{\beta-1/2}\int_{\eps t_0}^{t}F(V_{u/\eps})u^{-\beta}\dd u.
	\end{aligned}\]
	By self-similarity, $B^{(\eps)}:=(\sqrt{\eps}B_{t/\eps})_{t\geq0}$ has the same distribution as a standard Brownian motion.\\
	Assume that the convergence of the rescaled velocity process is proved in the strong way, that is
\begin{equation}\label{prob equiv}
\forall T\geq t_0, \ \sup_{\eps t_0\leq t \leq T} \abs{V_t^{(\eps)}-B_t^{(\eps)}}\stackrel{\mathbb{P}}{\longrightarrow}0,\ \mbox{ as }\eps\to 0.
\end{equation}
Then it suffices to prove that $g_{\eps}(B^{(\eps)})\stackrel{\text{}}{\Longrightarrow} g(\mathcal{B})$ and $\dd_u \left(g_{\eps}(V^{(\eps)}), g_{\eps}(B^{(\eps)}) \right)\stackrel{\mathbb{P}}{\longrightarrow}0$, as $\eps \to 0$ (see \citer[Theorem 3.1, p. 27,]{BillingsleyConvergenceprobabilitymeasures1999}).\\
On the one hand, the process $B^{(\eps)}$ being a 
Brownian motion and $\abs{\cdot}_{\mathbb{R}^2}$ denoting a norm on $\mathbb{R}^2$, the first convergence follows from 

\begin{equation}\label{prob equiv1}
\forall T\geq t_0 ,\ \sup_{\eps t_0\leq t \leq T} \abs{g_{\eps}(\mathcal{B}_t)-g(\mathcal{B}_t)}_{\mathbb{R}^2}\stackrel{\mathbb{P}}{\longrightarrow}0,\ \mbox{ as }\eps\to 0.
\end{equation}
Let us prove \eqref{prob equiv1}.
For every $\eps t_0\leq t\leq T$, we get
\[\begin{aligned}
	\abs{g_{\eps}(\mathcal{B}_t)-g(\mathcal{B}_t)}_{\mathbb{R}^2} &= \abs{\int_{0}^{\eps t_0}\mathcal{B}_s\dd s} \\ &\leq \int_0^{\eps t_0}\abs{\mathcal{B}_s}\dd s.
\end{aligned}
\]
Hence,
\[\mathbb{E}\left[ \sup_{\eps t_0\leq t \leq T} \abs{g_{\eps}(\mathcal{B}_t)-g(\mathcal{B}_t)}_{\mathbb{R}^2} \right] \leq \int_0^{\eps t_0}\mathbb{E}\abs{\mathcal{B}_s}\dd s \leq C \int_0^{\eps t_0}\sqrt{s}\dd s \underset{\eps \to 0}{\longrightarrow}0.
\]
On the other hand, we prove that 
\begin{equation}\label{prob equiv2}
\forall T\geq  t_0 ,\ \sup_{\eps t_0\leq t \leq T} \abs{g_{\eps}(V^{(\eps)}_t)- g_{\eps}(B_t^{(\eps)})}_{\mathbb{R}^2} \stackrel{\mathbb{P}}{\longrightarrow}0,\ \mbox{ as }\eps\to 0.
\end{equation}
For every $\eps t_0\leq t \leq T$, using \eqref{prob equiv}
\[
\begin{aligned}
\abs{g_{\eps}(V^{(\eps)}_t)- g_{\eps}(B_t^{(\eps)})}_{\mathbb{R}^2}  & =  \abs{V^{(\eps)}_t-B_t^{(\eps)} }+ \abs{\int_{\eps t_0}^t V_s^{(\eps)}-B_s^{(\eps)}\dd s}\\& 
\leq (1+ T-\eps t_0  ) \sup_{\eps t_0\leq t \leq T}\abs{V^{(\eps)}_t-B_t^{(\eps)}}\stackrel{\mathbb{P}}{\longrightarrow}0.
\end{aligned}
\]
	\item Let us prove now \eqref{prob equiv}. Recall that under both hypothesis \eqref{hyp1_levy} and \eqref{hyp2_levy}, there exists a positive constant $K$, such that $(\sqrt{\eps})^{\gamma}\abs{F\left(\dfrac{V_u^{(\eps)}}{\sqrt{\eps}}\right)}\leq K\abs{V_u^{(\eps)}}^{\gamma}$. 
	Modifying the factor in front of the integral part, we get
	\begin{equation*}\label{eq:V and B}
	V_t^{(\eps)} =\sqrt{\eps}(v_0 - B_{t_0})+ \sqrt{\eps}B_{t/\eps} - \eps^{\beta-\nicefrac{(\gamma+1)}{2}}\int_{\eps t_0}^{t}(\sqrt{\eps})^{\gamma}F\left(\dfrac{V_u^{(\eps)}}{\sqrt{\eps}}\right)u^{-\beta}\dd u.
	\end{equation*}
	\\
	It follows that, for all $ t_0\leq T$,
	\[ \begin{aligned}
	\sup_{\eps t_0\leq t \leq T} \abs{V_t^{(\eps)}-B_t^{(\eps)}}\leq & \sqrt{\eps}\abs{v_0-B_{t_0}}+\eps^{\beta-\nicefrac{(\gamma+1)}{2}}\sup_{\eps t_0\leq t \leq T}\abs{ \int_{\eps t_0}^{t}(\sqrt{\eps})^{\gamma}F\left(\dfrac{V_u^{(\eps)}}{\sqrt{\eps}}\right)u^{-\beta}\dd u }\\ \leq & \sqrt{\eps}\abs{v_0-B_{t_0}}+\eps^{\beta-\nicefrac{(\gamma+1)}{2}} \int_{\eps t_0}^{T}K\abs{V_u^{(\eps)}}^{\gamma}u^{-\beta}\dd u.
	\end{aligned} 
	\]
	Taking the expectation and using moment estimates (\cref{esperance}), we obtain, when $\beta \neq \frac{\gamma}{2}+1$ and since $\beta >\frac{\gamma+1}{2}$,
	\[\begin{aligned}
	\eps^{\beta-\nicefrac{(\gamma+1)}{2}} \mathbb{E}\left[\int_{\eps t_0}^{T}K\abs{V_u^{(\eps)}}^{\gamma}u^{-\beta}\dd u\right]&=  \eps^{\beta-\nicefrac{(\gamma+1)}{2}}\int_{\eps t_0}^{T}K\mathbb{E}\left[\abs{V_u^{(\eps)}}^{\gamma}\right]u^{-\beta}\dd u \\ &\leq  \eps^{\beta-\nicefrac{(\gamma+1)}{2}}\int_{\eps t_0}^{T}KC_{\gamma,\beta,t_0}u^{\frac{\gamma}{2}-\beta}\dd u\\& \leq  C\left( \eps^{\beta-\nicefrac{(\gamma+1)}{2}} T^{\frac{\gamma}{2}-\beta+1}-t_0^{\frac{\gamma}{2}-\beta+1}\sqrt{\eps}\right) \underset{\eps \to 0}{\longrightarrow} 0.
	\end{aligned} \]
	Hence, setting $r= \min(\frac{1}{2}, \beta-\nicefrac{(\gamma+1)}{2})>0$  \[\mathbb{E}\left[ \sup_{\eps t_0\leq t \leq T} \abs{V_t^{(\eps)}-B_t^{(\eps)}} \right] = O(\eps^{r}).\]
	The case $\beta=\frac{\gamma}{2}+1$ can be treated similarly to get 
	\[\mathbb{E}\left[ \sup_{\eps t_0\leq t \leq T} \abs{V_t^{(\eps)}-B_t^{(\eps)}} \right] = O(\sqrt{\eps}\ln(\eps)).\]
	This concludes the proof.
\end{steps}
\end{proof}
\begin{remark}
	One can observe that the only moment in this proof, when we need the condition "$\gamma<1$ or for all $v \in \mathbb{R},\ vF(v)$" is when we are proving the moment estimates. 
\end{remark}
\subsection{Asymptotic behaviour in the critical regime under \eqref{hyp1_levy}} \label{critical}
Assume in this section that $\beta= \frac{\gamma+1}{2}$ and \eqref{hyp1_levy} is satisfied.  
\begin{proof}[Proof of \cref{critique}]
\begin{steps}
\item
As in the first step of the previous section, it suffices to prove the convergence of the rescaled velocity process $(\sqrt{\eps}V_{t/\eps})_t$. Keeping same notations, we prove that $g_{\eps}(V^{(\eps)})$ converges in distribution in $\mathcal{C}([0,+\infty))$ to $g(\mathcal{V})$. In order to see $V^{(\eps)}$ as a process of $\mathcal{C}([0,+\infty))$, let us set for all $s\in [0,\eps t_0]$, $V_s^{(\eps)}:=V_{\eps t_0}^{(\eps)}=\sqrt{\eps}v_0 $. Call $P_{\eps}$, $P$ the distribution of $V^{(\eps)}$, $\mathcal{V}$ respectively. Then, using Pormanteau theorem (see \citer[Theorem 2.1 p.16]{BillingsleyConvergenceprobabilitymeasures1999}), it suffices to prove that for all function $h:{\mathcal{C}}([0,+\infty))\times{\mathcal{C}}([0,+\infty)) \to \mathbb{R}$ bounded and uniformly continuous, 
\[\int_{\mathcal{C}([0,+\infty))^2}h(g_{\eps}(\omega))\dd P_{\eps}(\dd \omega) \underset{\eps \to 0}{\longrightarrow }\int_{\mathcal{C}([0,+\infty))^2}h(g(\omega))\dd P(\dd \omega).  \]
Take a bounded and uniformly continuous function $h$.
By assumption, one knows that $P_{\eps}\Longrightarrow P$, hence, by \citer[Problem 4.12 p. 64,]{KaratzasBrownianMotionStochastic1998}, it suffices to prove that the uniformly bounded sequence $(h\circ g_{\eps})$ of continuous functions on $\mathcal{C}([0,+\infty))$ converges uniformly on compact subsets of $\mathcal{C}([0,+\infty))$ to the continuous function $h\circ g$.
Let $K$ be a compact set of $\mathcal{C}([0,+\infty))$. Then, for all $\omega\in K$, $\max_{[0,\eps t_0]}\abs{\omega}$  is uniformly bounded by a constant, called $M$. 
\\
Fix $\eta>0$. By the uniform continuity of $h$, there exists $\delta>0$ such that for all $\omega \in K$,
\[\dd_u(g_{\eps}(\omega),g(\omega))\leq \delta \Rightarrow \abs{h\circ g_{\eps}(\omega),h\circ g (\omega)}\leq \eta.  \]
However, there exists $\eps_1>0$ small enough, such that for all $\eps\leq \eps_1$, for all $\omega\in K$, 
\[\dd_u(g_{\eps}(\omega),g(\omega)) \leq C \abs{\int_0^{\eps t_0}\omega(s)\dd s}\leq C\eps t_0 M \leq \delta. \]
\item
We first prove the f.d.d.\ convergence.
The exponential scaling process $V^{(e)}$ satisfies the time-homogeneous equation
\begin{equation}\label{equation_H}
\dd V_s^{(e)}=\dd W_s-\dfrac{V_s^{(e)}}{2}\dd s-F\big(V_s^{(e)}\big)\dd s,
\end{equation} where $(W_t)_{t\geq0}$ is a standard Brownian motion.\\
Using the bijection induced by the exponential change of time (\cref{changeoftime}), we get
\[ \left(\dfrac{V_{t_0e^t}}{\sqrt{t_0}e^{t/2}} \right)_{t\geq 0}= (H_t)_{t\geq 0},\]
as solutions of the same SDE, starting at the same point.
This can also be written as \begin{equation*}\label{equal in law}
\left(\dfrac{V_{t}}{\sqrt{t}} \right)_{t\geq t_0}= (H_{\log(t/t_0)})_{t\geq t_0}.
\end{equation*}
So, we have, for all $\eps>0$, and $(t_1, \cdots, t_d)\in [\eps t_0,+\infty)^d$,
\begin{equation}\label{eq: V and H}
	\left(\dfrac{V_{\eps^{-1}t_1}}{\sqrt{\eps^{-1}t_1}}, \cdots, \dfrac{V_{\eps^{-1}t_d}}{\sqrt{\eps^{-1}t_d}} \right)= \left( H_{\log(t_1)+\log((\eps t_0)^{-1})}, \cdots, H_{\log(t_d)+\log((\eps t_0)^{-1})} \right).
\end{equation}
As in \cite{GradinaruExistenceasymptoticbehaviour2013}, the scale function and the speed measure of $H$ are respectively
\[{\tt p}(x):=\int_{0}^x \exp\left(\dfrac{y^2}{2}+\dfrac{2}{\gamma+1}\sgn(y)F(\sgn(y))\abs{y}^{\gamma+1}\right) \dd y \]
and 
\[\nu_F(\dd x):= \exp\left(-\dfrac{x^2}{2}-\dfrac{2}{\gamma+1}\sgn(x)F(\sgn(x))\abs{x}^{\gamma+1}\right)\dd x.  \]
By the ergodic theorem (\citer[Theorem 23.15 p. 465]{KallenbergFoundationsModernProbability2002}), $H$ is $\Lambda_{F}$-ergodic, where $\Lambda_{F}$ is the probability measure associated to $\nu_F$.
Call $\widetilde{H}$ the solution of the time homogeneous equation \eqref{equation_H} such that the initial condition $\widetilde{H}_{-\infty}$ has the distribution $\Lambda_F$.\\
For $t_1, \cdots, t_d \in [\eps t_0, +\infty)^d$, let $\Lambda_{F,t_1, \cdots, t_d}:= \mathcal{L}(\widetilde{H}_{t_1}, \cdots, \widetilde{H}_{t_d})$ be the distribution of $(\widetilde{H}_{t_1}, \cdots, \widetilde{H}_{t_d})$. Then, for all $s\geq 0$, $\Lambda_{F,t_1, \cdots, t_d}= \Lambda_{F,t_1+s, \cdots, t_d+s}$. Indeed, thanks to the invariance property of $\Lambda_F$, $(\widetilde{H}_t)_{t\in \mathbb{R}}$ and $(\widetilde{H}_{t+s})_{t\in \mathbb{R}}$ satisfy the same SDE, starting at the same distribution. As a consequence, for all $\eps>0$,
\begin{equation}\label{eq: stationnarity}
	\mathcal{L}\left( \widetilde{H}_{\log(t_1)+\log((\eps t_0)^{-1})}, \cdots, \widetilde{H}_{\log(t_d)+\log((\eps t_0)^{-1})} \right)= \Lambda_{F,\log(t_1), \cdots, \log(t_d)}.
\end{equation} 
Moreover, by exponential ergodicity, for every $\psi: \mathbb{R}^d\to \mathbb{R}$ continuous and bounded function, we can prove that 
\begin{equation} \label{eq: ergodicity}
    \abs{\mathbb{E}\left[\psi\left(H_{\log(t_1/(t_0\eps))},\cdots, H_{\log(t_d/(t_0\eps))} \right)  \right]-\mathbb{E}\left[\psi\left(\widetilde{H}_{\log(t_1/(t_0\eps))},\cdots,  \widetilde{H}_{\log(t_d/(t_0\eps))} \right) \right] } \underset{\eps \to 0}{\longrightarrow}0.
\end{equation}
We postpone the proof of this convergence in Step 3.
\\To conclude this step, gather \eqref{eq: V and H}, \eqref{eq: stationnarity} and \eqref{eq: ergodicity} to get
\[\left(\dfrac{V_{\eps^{-1}t_1}}{\sqrt{\eps^{-1}t_1}}, \cdots, \dfrac{V_{\eps^{-1}t_d}}{\sqrt{\eps^{-1}t_d}} \right)\underset{\eps\to 0}{\Longrightarrow} \Lambda_{F,\log(t_1), \cdots, \log(t_d)}. \]
This can be written as
\[\left(\sqrt{\eps}V_{t_1/\eps}, \cdots, \sqrt{\eps}V_{t_d/\eps} \right)\underset{\eps\to 0}{\Longrightarrow} T*\Lambda_{F,\log(t_1), \cdots, \log(t_d)}, \]
where $T*\Lambda_{F,\log(t_1), \cdots, \log(t_d)}$ is the pushforward of the measure $\Lambda_{F,\log(t_1), \cdots, \log(t_d)}$ by the linear map $T(u_1, \cdots, u_d):= (\sqrt{t_1}u_1, \cdots, \sqrt{t_d}u_d)$.
\item Let us now prove \eqref{eq: ergodicity}. 
Pick $\eps t_0 \leq s\leq t$. Set $h_0=v_0\sqrt{t_0}^{-1}$.
Actually we prove a more general result, which will also be useful in the last regime. The convergence \eqref{eq: ergodicity} will be a direct consequence of this lemma. 
\begin{lem}\label{lem: ergodic_fdd}
	 Let $H$ be an exponential ergodic process with invariant measure $\nu$, solution to a SDE driven by a Brownian motion. Pick a continuous function $\phi: [t_0, +\infty)\to \mathbb{R}$ satisfying $\lim_{s\to +\infty}\phi(s)=+\infty$.\\
	Then, for all integer $d\geq 1$, every continuous and bounded function $\psi: \mathbb{R}^d\to \mathbb{R}$, all $h_0\in \mathbb{R}$ and all $t_1,\cdots, t_d \in [\eps t_0,+\infty)^d$,
	\begin{equation*} 
		\abs{\mathbb{E}\left[\psi\left(H_{\phi(\eps^{-1}t_1)},\cdots, H_{\phi(\eps^{-1}t_d)} \right) \Big | H_0=h_0 \right]-\mathbb{E}\left[\psi\left(H_{\phi(\eps^{-1}t_1)},\cdots, H_{\phi(\eps^{-1}t_d)} \right)\Big | H_0\sim \nu \right] } \underset{\eps \to 0}{\longrightarrow}0.
	\end{equation*}
\end{lem}
\begin{proof}
For the sake of clarity, let us give a proof for $d=2$. The general case $d\geq2$ is similar.\\
Let $\psi: \mathbb{R}^2\to \mathbb{R}$ be a continuous and bounded function. 

We set $\mu_{\eps}:= \mathcal{L}\left(H_{\phi(\eps^{-1}s)} \Big| H_0=h_0 \right)$. 
We use the generalized Markov property of solution to SDE driven by Brownian motion process (see \citer[Theorem 21.11 p. 421]{KallenbergFoundationsModernProbability2002}).
This leads to  
\[
        \mathbb{E}\left[\psi\left(H_{\phi(\eps^{-1}s)}, H_{\phi(\eps^{-1}t)} \right) \Big | H_0=h_0 \right]
        = \mathbb{E}\left[\psi\left(H_{0}, H_{\phi(\eps^{-1}t)-\phi(\eps^{-1}s)} \right)\Big | H_0\sim \mu_{\eps }\right]
  \]
  and, since $\Lambda_F$ is invariant,
  \[
    \mathbb{E}\left[\psi\left(H_{\phi(\eps^{-1}s)}, H_{\phi(\eps^{-1}t)} \right) \Big | H_0\sim \nu \right]
    = \mathbb{E}\left[\psi\left(H_{0}, H_{\phi(\eps^{-1}t)-\phi(\eps^{-1}s)} \right)\Big | H_0\sim \nu\right].
\]
Then, we are reduced to prove 
\begin{equation*}
\abs{\mathbb{E}\left[\psi\left(H_{0}, H_{\phi(\eps^{-1}t)-\phi(\eps^{-1}s)} \right)\Big | H_0\sim \mu_{\eps }\right]-\mathbb{E}\left[\psi\left(H_{0}, H_{\phi(\eps^{-1}t)-\phi(\eps^{-1}s)} \right)\Big | H_0\sim \nu\right]}\underset{\eps \to 0}{\longrightarrow}0.
\end{equation*} 
    Hence, setting $p(t,x,\dd y):= \mathbb{P}_x(H_t\in \dd y)$ and $\|.\|_{TV}$ for the total variation norm, we get
    \begin{multline*}
        \abs{\mathbb{E}\left[\psi\left(H_{0}, H_{\phi(\eps^{-1}t)-\phi(\eps^{-1}s)} \right)\Big | H_0\sim \mu_{\eps }\right]-\mathbb{E}\left[\psi\left(H_{0}, H_{\phi(\eps^{-1}t)-\phi(\eps^{-1}s)} \right)\Big | H_0\sim \nu\right]}\\ \leq \abs{\int_{\mathbb{R}}\mathbb{E}\left[\psi\left(H_{0}, H_{\phi(\eps^{-1}t)-\phi(\eps^{-1}s)} \right)\Big | H_0=y\right]\left(\mu_{\eps}(\dd y)- \nu(\dd y) \right)}
		\\ \leq \left\|\psi \right\|_{\infty}\int_{\mathbb{R}}\abs{p\left(\phi(\eps^{-1}s),h_0, \dd y\right)-\nu(\dd y)} \\ \leq \left\|\psi \right\|_{\infty} \|p\left(\phi(\eps^{-1}s),h_0, \cdot\right)-\nu\|_{TV}.
    \end{multline*}
    We let $\eps \to 0$, using the exponential ergodicity of $H$.
\end{proof}
\item Let us prove now the tightness of the family of laws of continuous process $\left(V^{(\eps)}\right)_{t\geq \eps t_0}=\left(\sqrt{\eps} V_{t/\eps} \right)_{t\geq \eps t_0}$ on every compact interval $[m,M]$, $0<m\leq M$. We prove the Kolmogorov criterion stated in \citer[Problem 4.11 p. 64,] {KaratzasBrownianMotionStochastic1998}.\\
Take $\eps_0$ small enough such that for all $\eps\leq \eps_0$, $\eps t_0 \leq m$. Fix $m\leq s\leq t \leq M$ and $\alpha>2$. Recalling that $B^{(\eps)}$ is a Brownian motion, using Jensen's inequality, moment estimates (\cref{esperance}) and the relation $\beta=\frac{\gamma+1}{2}$, we can write
\[
	\begin{aligned}
		\mathbb{E}\left[\abs{V_t^{(\eps)}-V_s^{(\eps)}}^{\alpha}  \right] &\leq C_{\alpha}\mathbb{E}\left[\abs{B_t^{(\eps)}-B_s^{\eps} }^{\alpha} \right] +C_{\alpha}\mathbb{E}\left[\abs{\sqrt{\eps}\int_{s/\eps}^{t/\eps}F(V_u)u^{-\beta}\dd u }^{\alpha} \right]
		\\ & \leq C_{\alpha} \mathbb{E}\left[\abs{B_t-B_s }^{\alpha} \right] +C_{\alpha} \eps^{1-\frac{\alpha}{2}}(t-s)^{\alpha-1}\mathbb{E}\left[\int_{s/\eps}^{t/\eps}\abs{F(V_u)}^{\alpha}u^{-\beta\alpha}\dd u  \right] \\ & \leq C_{\alpha} \mathbb{E}\left[\abs{B_{t-s} }^{\alpha} \right] +C_{\alpha}\eps^{1-\frac{\alpha}{2}}(t-s)^{\alpha-1}\int_{s/\eps}^{t/\eps}u^{\frac{\gamma\alpha}{2}-\beta\alpha}\dd u   \\ & \leq C_{\alpha} (t-s)^{\frac{\alpha}{2}}+C_{\alpha}\eps^{1-\frac{\alpha}{2}}(t-s)^{\alpha-1}\int_{s/\eps}^{t/\eps}u^{-\frac{\alpha}{2}}\dd u  \\ & \leq C_{\alpha} (t-s)^{\frac{\alpha}{2}}+C_{\alpha}(t-s)^{\alpha-1}(t^{1-\frac{\alpha}{2}}-s^{1-\frac{\alpha}{2}}) \\ &\leq C_{\alpha} (t-s)^{\frac{\alpha}{2}}+ C_{\alpha,m,M}(t-s)^{\alpha-1} \\ & \leq C_{\alpha,m,M}(t-s)^{\frac{\alpha}{2}}.
	\end{aligned}
\]
Since $\alpha>2$, then $\frac{\alpha}{2}>1$ and the upper bound does not depend on $\eps$.
Furthermore, by moment estimates (\cref{esperance}),
\[\sup_{\eps\leq \eps_0}\mathbb{E}\left[\abs{V_m^{(\eps)}} \right] \leq \sqrt{m }<+\infty. \]
\item[\textsc{Conclusion.}] The previous steps yields weak convergence on every compact set (\citer[Theorem 13.1 p. 139,]{BillingsleyConvergenceprobabilitymeasures1999}). The conclusion follows from \citer[Theorem 16.7 p. 174,]{BillingsleyConvergenceprobabilitymeasures1999}, since all processes considered are continuous.
\end{steps}
\end{proof}
\begin{ex}
	We will see that the limiting process $\mathcal{V}$ is more explicit in the linear case ($\gamma=1$). Choose $F(1)=1$, $F(-1)=-1$, the process $\widetilde{H}$ solution of \eqref{equation_H} is in fact an Ornstein Uhlenbeck process with invariant measure $\Lambda_F(dx):=e^{-\frac{3x^2}{2}}\dd x$. It is a centered Gaussian process, hence for all $s_1,\cdots, s_d$, its f.d.d. $\Lambda_{F, s_1, \cdots, s_d}$ are Gaussian. As a consequence, knowing the covariance function $K$ is enough to provide the law of the process. 
	Since $\widetilde{H}$ is a stationary Ornstein-Uhlenbeck process, one has $K:s,t\mapsto \frac{1}{3} e^{-\frac{3}{2}\abs{t-s}}$.
	Hence, the limiting process $\mathcal{V}$ having f.d.d $T*\Lambda_{F,\log(t_1), \cdots, \log(t_d)}$ is a centered Gaussian process with covariance function $s, t\mapsto \frac{1}{3}\frac{(s\wedge t)^2}{s\vee t}$.	
\end{ex}
\subsection{Asymptotic behaviour in the subcritical regime under \eqref{hyp1_levy}}
Assume in this section that $\beta<\frac{\gamma+1}{2}$ and $F:v\mapsto \rho\sgn(v)\abs{v}^{\gamma}$ with $\gamma\geq 1$. For simplicity, we shall write $\varphi$ instead of $\varphi_q$.
\begin{proof}[Proof of \cref{sub-critical}]
\begin{steps}
\item We first prove the f.d.d.\ convergence of the velocity process $(V_t^{(\eps)})_{t\geq \eps t_0}:= (\eps^q V_{t/\eps})_{t\geq \eps t_0}$. Again we give a proof only for $d=2$, since the general case $d\geq2$ is similar. \\
The power scaling process $V^{(q)}$, solution to \eqref{power hyp2} satisfies 
\begin{equation*}
	\dd V_s^{(q)} =\dd W_s- F\Big(V_s^{(q)}\Big)\dd s-q \varphi^{2q-1}(s)V_s^{(q)}\dd s.
\end{equation*}
We call $H$ the ergodic process solution to the SDE
\begin{equation}\label{SDE_homogene}
	\dd H_s =\dd W_s- F\Big(H_s\Big)\dd s,\mbox{ with } H_0=h_0:=v_0 t_0^{-q}.
\end{equation}
We denote by $\Pi_F(\dd x):= e^{-\frac{2\rho}{\gamma+1}\abs{x}^{\gamma+1}}\dd x$ its invariant measure.
Using the bijection induced by the power change of time (\cref{changeoftime}), as solutions of the same SDE starting at the same point, 
we have, for all $\eps>0$, and $s,t\in [\eps t_0,+\infty)^2$,
\begin{equation*}
\left(\eps^q \dfrac{V_{\eps^{-1}s}}{s^q}, \eps^q \dfrac{V_{\eps^{-1}t}}{t^q}   \right)= \left(V^{(q)}_{\varphi^{-1}(\eps^{-1}s)},  V^{(q)}_{\varphi^{-1}(\eps^{-1}t)} \right).
\end{equation*}
Using \citer[Theorem 3.1 p. 27,]{BillingsleyConvergenceprobabilitymeasures1999}, it suffices to prove that for all $s,t\in [\eps t_0,+\infty)^2$,
\begin{itemize}
	\item $\abs{\left(H_{\varphi^{-1}(\eps^{-1}s)},H_{\varphi^{-1}(\eps^{-1}t)} \right)-\left(V^{(q)}_{\varphi^{-1}(\eps^{-1}s)},V^{(q)}_{\varphi^{-1}(\eps^{-1}t)} \right) }_{\mathbb{R}^2}\underset{\eps \to 0}{\longrightarrow}0$.
	\item $\left(H_{\varphi^{-1}(\eps^{-1}s)},H_{\varphi^{-1}(\eps^{-1}t)} \right)\underset{\eps \to 0}{\Longrightarrow} \Pi_F \otimes \Pi_F $.
\end{itemize}
\item We prove that for all $t\geq\eps t_0$, $\mathbb{E}\left[\left(H_{\varphi^{-1}(\eps^{-1}t)}- V^{(q)}_{\varphi^{-1}(\eps^{-1}t)} \right)^2 \right] \underset{\eps\to0}{\longrightarrow} 0$.
\\
Pick $t\geq \eps t_0$. For simplicity of notation, we write $H^{(\varphi,\eps)}_t:=H_{\varphi^{-1}(\eps^{-1}t)}$ and $V^{(\varphi,\eps)}_t:=V^{(q)}_{\varphi^{-1}(\eps^{-1}t)}$.
We have
\begin{equation*}
	\dd \left(H^{(\varphi,\eps)}_t- V^{(\varphi,\eps)}_t\right)= -\eps^{2q-1}\left(F(H^{(\varphi,\eps)}_t)-F(V^{(\varphi,\eps)}_t) \right)t^{-2q}\dd t+ qt^{-1}V^{(\varphi,\eps)}_t\dd t.
\end{equation*}
Pick $\delta>0$. By straightforward differentiation, we can write
\begin{multline*}
	\dd \left(H^{(\varphi,\eps)}_t- V^{(\varphi,\eps)}_t\right)^2 \leq  
	-\dfrac{2\eps^{2q-1}}{t^{2q}}\left(F(H^{(\varphi,\eps)}_t)-F(V^{(\varphi,\eps)}_t) \right)\left(H^{(\varphi,\eps)}_t- V^{(\varphi,\eps)}_t\right)\mathbb{1}_{\abs{H^{(\varphi,\eps)}_t- V^{(\varphi,\eps)}_t}>\delta}\dd t\\+2t^{-1}qV^{(\varphi,\eps)}_t\left(H^{(\varphi,\eps)}_t- V^{(\varphi,\eps)}_t\right)\dd t.
\end{multline*}
Since $\gamma\geq1$, the function $F^{-1}$ is $\frac{1}{\gamma}$-Hölder, therefore there exists $C_{\gamma}>0$ such that, 
\begin{multline}\label{eq: ODE_V_q}
	\dd \left(H^{(\varphi,\eps)}_t- V^{(\varphi,\eps)}_t\right)^2 \leq  
	-\dfrac{2\eps^{2q-1}}{t^{2q}}C_{\gamma}\delta^{\gamma-1}\left(H^{(\varphi,\eps)}_t- V^{(\varphi,\eps)}_t\right)^2\mathbb{1}_{\abs{H^{(\varphi,\eps)}_t- V^{(\varphi,\eps)}_t}>\delta}\dd t\\+2t^{-1}qV^{(\varphi,\eps)}_t\left(H^{(\varphi,\eps)}_t- V^{(\varphi,\eps)}_t\right)\dd t.
\end{multline}
 We set $g_{\eps}(t)=\mathbb{E}\left[\left(H^{(\varphi,\eps)}_t- V^{(\varphi,\eps)}_t
\right)^2 \right]$ and $\widetilde{g}_{\eps}(t)=\mathbb{E}\left[\left(H^{(\varphi,\eps)}_t- V^{(\varphi,\eps)}_t
\right)^2\mathbb{1}_{\abs{H^{(\varphi,\eps)}_t- V^{(\varphi,\eps)}_t}>\delta} \right]$. 
Taking expectation in \eqref{eq: ODE_V_q}, we get

\begin{equation}\label{ODE}
\widetilde{g}_{\eps}'(t)\leq -\dfrac{2\eps^{2q-1}}{t^{2q}}C_{\gamma}\delta^{\gamma-1} \widetilde{g}_{\eps}(t)+b_{\eps}(t), \mbox{ with } \widetilde{g}_{\eps}(\eps t_0)=0
\end{equation}
where
\[b_{\eps}(t):=  2t^{-1}q\mathbb{E}\left[V^{(\varphi,\eps)}_t\left(H^{(\varphi,\eps)}_t- V^{(\varphi,\eps)}_t\right)\right].\]
Using Cauchy-Schwarz inequality and moment estimates (\cref{esperance}), we have
\[	\abs{b_{\eps}(t)}\leq 2t^{-1}\abs{q}\sqrt{\mathbb{E}\left[\left(V^{(\varphi,\eps)}_t\right)^2\right]}\sqrt{g_{\eps}(t)} \leq C 2t^{-1}\abs{q}\sqrt{\eps^{2q-1}t^{1-2q}g_{\eps}(t)}.
\]

Set $h(t):= \dfrac{1}{1-2q}C_{\gamma}\delta^{\gamma-1}t^{1-2q}$.
We use the comparison theorem for ordinary differential equation on \eqref{ODE} to get
\begin{equation*}\label{solution_ODE}
\widetilde{g}_{\eps}(t)\leq  \int_{\eps t_0}^t b_{\eps}(s)\exp(-2\eps^{2q-1}\left(h(t)-h(s)\right))\dd s.
\end{equation*}
As a consequence, we deduce that
\[\begin{aligned}
	g_{\eps}(t)&\leq \delta^2+\widetilde{g}_{\eps}(t)
	\\ & \leq \delta^2+ \exp(-2\eps^{2q-1}h(t))C\int_{\eps t_0}^t 2s^{-1} \sqrt{\eps^{2q-1}s^{1-2q}}\sqrt{g_{\eps}(s)\exp(2\eps^{2q-1}h(s))} \exp(\eps^{2q-1}h(s))\dd s.
\end{aligned}\]
Applying a Gronwall-type lemma (\cref{gronwall}) to the function $g_{\eps}\exp(2\eps^{2q-1}h)$, we obtain
\begin{equation*}
	\begin{aligned}
		g_{\eps}(t)&\leq C\delta^2+C\left(\int_{\eps t_0}^t s^{-1} \sqrt{\eps^{2q-1}s^{1-2q}} \exp(-\eps^{2q-1}(h(t)-h(s)))\dd s\right)^2.
	\end{aligned}
\end{equation*}
We conclude, using the dominated convergence theorem, since $1-2q>0$, that for all $\delta>0$
\begin{equation}\label{limsup}
	0\leq \limsup_{\eps\to 0}g_{\eps}(t)\leq \delta^2.
\end{equation}
To prove the domination hypothesis, notice that by optimization of the function $x\mapsto \sqrt{x}\exp(-Ax)$,
\[	\mathbb{1}_{\eps t_0\leq s \leq t}s^{-1} \sqrt{\eps^{2q-1}s^{1-2q}} \exp(-\eps^{2q-1}(h(t)-h(s))) \leq \mathbb{1}_{0\leq s \leq t} s^{-\frac{1}{2}-q} \dfrac{1}{\sqrt{h(t)-h(s)}}.
\]
This function is integrable, since $1-2q>0$.\\
We let $\delta\to 0$ in \eqref{limsup} to conclude that for all $t>0$, $\lim_{\eps\to 0}g_{\eps}(t)=0$.
\item Pick $s,t\in [\eps t_0,+\infty)^2$. We prove that the solution $H$ to \eqref{SDE_homogene} satisfies
\begin{equation}\label{decouplage_H}
	\left(H_{\varphi^{-1}(\eps^{-1}s)},H_{\varphi^{-1}(\eps^{-1}t)} \right)\underset{\eps \to 0}{\Longrightarrow} \Pi_F \otimes \Pi_F. 
\end{equation}
Observe that
\begin{equation}\label{eq: difference_varphi}
	\varphi^{-1}(\eps^{-1}t)-\varphi^{-1}(\eps^{-1}s)= \dfrac{t^{1-2q}-s^{1-2q}}{\eps^{1-2q}} \underset{\eps \to 0}{\longrightarrow}0.
\end{equation} 
By \cref{lem: ergodic_fdd}, for every continuous and bounded function $\varphi$, we can write
\begin{equation*}
    \abs{\mathbb{E}\left[\psi\left(H_{\varphi^{-1}(\eps^{-1}s)}, H_{\varphi^{-1}(\eps^{-1}t)} \right) \Big | H_0=h_0 \right]-\mathbb{E}\left[\psi\left(H_{\varphi^{-1}(\eps^{-1}s)}, H_{\varphi^{-1}(\eps^{-1}t)} \right) \Big | H_0\sim \Pi_F \right] } \underset{\eps \to 0}{\longrightarrow}0.
\end{equation*}
Hence, it suffices to prove that for every bounded continuous functions $f,g:\mathbb{R}\to \mathbb{R}$, the following convergence holds
\begin{equation*}
	\lim_{\eps\to 0}\mathbb{E}\left[f\left(H_{\varphi^{-1}(\eps^{-1}s)}\right)g\left(H_{\varphi^{-1}(\eps^{-1}t)} \right) \Big | H_0\sim \Pi_F \right]= \Pi_F(f)\Pi_F(g). 
\end{equation*}
The following reasoning is inspired from \citer[the proof of Lemma 3.2 p. 7-8]{CattiauxAsymptoticanalysisdiffusion2010}. 
Since $H_0$ is starting from the invariant measure, up to considering $f-\Pi_F(f)$ and $g-\Pi_F(g)$, we can assume that $f$ and $g$ have zero $\Pi_F$-mean.
We call $(P_t)_{t\geq0}$ the semigroup of $H$, then we get, by invariance property of $\Pi_F$,
\[\begin{aligned}
	\mathbb{E}\left[f\left(H_{\varphi^{-1}(\eps^{-1}s)}\right)g\left(H_{\varphi^{-1}(\eps^{-1}t)} \right) \Big | H_0\sim \Pi_F \right]&= \int P_{\varphi^{-1}(\eps^{-1}s)}\left( fP_{\varphi^{-1}(\eps^{-1}t)-\varphi^{-1}(\eps^{-1}s)}g\right) \dd \Pi_F\\&=\int fP_{\varphi^{-1}(\eps^{-1}t)-\varphi^{-1}(\eps^{-1}s)}g\dd \Pi_F.
\end{aligned}\]
Note that $U:v\mapsto \frac{\abs{v}^{1+\gamma}}{1+\gamma}$ is a convex function, thus a $\lambda$-Poincaré inequality holds for the process $H$ (see \cite{BobkovIsoperimetricAnalyticInequalities1999} p. 1904).
This implies the exponential decay of the variance (see \citer[Theorem 4.2.5 p. 183,]{BakryAnalysisGeometryMarkov2014}), i.e.\ there exists a constant $C>0$ such that, since $\Pi_F$ is a probability measure, 
\[ \begin{aligned}
	\abs{\int fP_{\varphi^{-1}(\eps^{-1}t)-\varphi^{-1}(\eps^{-1}s)}g\dd \Pi_F}&\leq \left\|fP_{\varphi^{-1}(\eps^{-1}t)-\varphi^{-1}(\eps^{-1}s)}g \right\|_2 \\ &\leq \left\|f \right\|_{\infty}\left\|P_{\varphi^{-1}(\eps^{-1}t)-\varphi^{-1}(\eps^{-1}s)}g \right\|_2 \\ & \leq C\left\|f \right\|_{\infty}\left\|g \right\|_{\infty}e^{-\lambda\left(\varphi^{-1}(\eps^{-1}t)-\varphi^{-1}(\eps^{-1}s)\right)}.
\end{aligned}\]
We deduce \eqref{decouplage_H} from \eqref{eq: difference_varphi}.
\item We prove the f.d.d.\ convergence of the position process $(X_t^{(\eps)})_{t\geq \eps t_0}:=(\eps^{\beta+\frac{1}{2}}X_{t/\eps})_{t\geq \eps t_0}$. Take $\gamma=1$ and $\beta\in (-\frac{1}{2},1)$. Pick $t\geq \eps t_0$. By Itô's formula applied to $t^{\beta}V_t$, we get
\begin{equation*}
	\rho X_t^{(\eps)}= \eps^{\beta+\frac{1}{2}}(t_0^{\beta}v_0+x_0)-\eps^{\frac{1-\beta}{2}}t^{\beta}V_t^{(\eps)}+\eps^{\beta+\frac{1}{2}}\int_{t_0}^{t/\eps}s^{\beta}\dd B_s +\eps^{\beta+\frac{1}{2}}\int_{t_0}^{t/\eps}\beta s^{\beta-1}V_s \dd s.
\end{equation*}
Since $\beta>-\frac{1}{2}$, the first term converges to 0 in probability as $\eps \to 0$.
Moreover, by Itô's formula, for all $t\geq t_0$,
\begin{equation*}
	\dfrac{\dd }{\dd t}\mathbb{E}\left[V^2_t\right]= - 
2\rho s^{-\beta} \mathbb{E}\left[V^2_s\right] +1.
\end{equation*}
Hence, by comparison theorem for ordinary differential equation, 
\begin{equation*}
	\mathbb{E}\left[V^2_t\right]\leq \exp(-2\rho\dfrac{t^{1-\beta}}{1-\beta})\left(v_0^2+\int_{t_0}^t \exp(2\rho\dfrac{s^{1-\beta}}{1-\beta})\dd s\right).
\end{equation*}
Using an integration by parts, we deduce that there exists a positive constant $C$ such that, for all $t\geq t_0$,
\begin{equation*}\label{eq: moment_V_lineaire}
	\mathbb{E}\left[V^2_t\right]\leq Ct^{\beta}.
\end{equation*}
As a consequence, we obtain 
\[\begin{aligned}
	\mathbb{E}\left[\abs{-\eps^{\frac{1-\beta}{2}}t^{\beta}V_t^{(\eps)}+\eps^{\beta+\frac{1}{2}}\int_{t_0}^{t/\eps}\beta s^{\beta-1}V_s \dd s} \right]& \leq \eps^{\frac{1-\beta}{2}}t^{\beta} \mathbb{E}\left[\abs{V_t^{(\eps)}}\right]+\eps^{\beta+\frac{1}{2}}\int_{t_0}^{t/\eps}\beta s^{\beta-1}\mathbb{E}\left[\abs{V_s}\right] \dd s
	\\& \leq C \eps^{\frac{1}{2}}t^{\frac{3\beta}{2}}+C\eps^{\frac{1-\beta}{2}}t^{\frac{3\beta}{2}}-C\eps^{\beta+\frac{1}{2}}t_0^{\frac{3\beta}{2}} \underset{\eps \to 0}{\longrightarrow}0.
\end{aligned}
	\]
It remains to study the centered Gaussian process $M^{(\eps)}_t:=\eps^{\beta+\frac{1}{2}}\int_{t_0}^{t/\eps}s^{\beta}\dd B_s$.
By Itô's isometry and since $\beta>-\frac{1}{2}$, for all $\eps t_0\leq s \leq t$, we can write
\[	\textrm{Cov}(M^{(\eps)}_s,M^{(\eps)}_t)=\eps^{2\beta+1}\int_{t_0}^{s/\eps}u^{2\beta}\dd s \underset{\eps \to 0}{\sim}\dfrac{ s^{1+2\beta}}{1+2\beta}.
\]
Since the convergence of centered Gaussian processes is characterized by the convergence of their covariance functions, the conclusion follows from \citer[Theorem 3.1, p. 27,]{BillingsleyConvergenceprobabilitymeasures1999}.
\end{steps}
\end{proof}

\noindent
\paragraph{Acknowledgements}

The authors would like to thank Jürgen Angst for valuable exchanges and suggestions about this work and Thomas Cavallazzi for his careful
reading of the manuscript. We would also like to thank the anonymous referees for her/his careful reading of the manuscript and useful advices.

\bibliographystyle{alpha}
\bibliography{kinetic_inh_gene}

\begin{thebibliography}{CCM10}

\bibitem[B{\'e}t21]{BethencourtStablelimittheorems2021}
Lo{\"i}c B{\'e}thencourt.
\newblock Stable limit theorems for additive functionals of one-dimensional
  diffusion processes.
\newblock {\em arXiv:2104.06027 [math]}, April 2021.

\bibitem[BGL14]{BakryAnalysisGeometryMarkov2014}
Dominique Bakry, Ivan Gentil, and Michel Ledoux.
\newblock {\em Analysis and {{Geometry}} of {{Markov Diffusion Operators}}},
  volume 348 of {\em Grundlehren Der Mathematischen {{Wissenschaften}}}.
\newblock {Springer International Publishing}, {Cham}, 2014.

\bibitem[Bil99]{BillingsleyConvergenceprobabilitymeasures1999}
Patrick Billingsley.
\newblock {\em Convergence of Probability Measures}.
\newblock Wiley Series in Probability and Statistics. {{Probability}} and
  Statistics Section. {Wiley}, 2nd ed edition, 1999.

\bibitem[Bob99]{BobkovIsoperimetricAnalyticInequalities1999}
S.~G. Bobkov.
\newblock Isoperimetric and {{Analytic Inequalities}} for {{Log-Concave
  Probability Measures}}.
\newblock {\em The Annals of Probability}, 27(4):1903--1921, 1999.

\bibitem[CCM10]{CattiauxAsymptoticanalysisdiffusion2010}
Patrick Cattiaux, Djalil Chafai, and S{\'e}bastien Motsch.
\newblock Asymptotic analysis and diffusion limit of the {{Persistent Turning
  Walker Model}}.
\newblock {\em Asymptotic Analysis}, 67(1\textendash 2):17--31, 2010.

\bibitem[CNP19]{CattiauxDiffusionlimitkinetic2019}
Patrick Cattiaux, Elissar Nasreddine, and Marjolaine Puel.
\newblock Diffusion limit for kinetic {{Fokker-Planck}} equation with heavy
  tails equilibria: {{The}} critical case.
\newblock {\em Kinetic \& Related Models}, 12(4):727, 2019.

\bibitem[EG15]{EonGaussianasymptoticsnonlinear2015}
Richard Eon and Mihai Gradinaru.
\newblock Gaussian asymptotics for a non-linear langevin type equation driven
  by an $\alpha$-stable lévy noise.
\newblock {\em Electron. J. Probab.}, 20, 2015.

\bibitem[FT21]{FournierOnedimensionalcritical2021}
Nicolas Fournier and Camille Tardif.
\newblock One dimensional critical {{Kinetic Fokker-Planck}} equations,
  {{Bessel}} and stable processes.
\newblock {\em Communications in Mathematical Physics}, 381, January 2021.

\bibitem[GL21]{GradinaruAsymptoticbehaviortimeinhomogeneous2021}
Mihai Gradinaru and Emeline Luirard.
\newblock Asymptotic behavior for a time-inhomogeneous stochastic differential
  equation driven by an {$\alpha$}-stable {{L\'evy}} process.
\newblock {\em arXiv:2112.07287 [math]}, December 2021.

\bibitem[GO13]{GradinaruExistenceasymptoticbehaviour2013}
Mihai Gradinaru and Yoann Offret.
\newblock Existence and asymptotic behaviour of some time-inhomogeneous
  diffusions.
\newblock {\em Ann. Inst. H. Poincar\'e Probab. Statist.}, 49(1):182--207,
  February 2013.

\bibitem[Kal02]{KallenbergFoundationsModernProbability2002}
Olav Kallenberg.
\newblock {\em Foundations of {{Modern Probability}}}.
\newblock Probability and {{Its Applications}}. {Springer New York}, {New York,
  NY}, 2002.

\bibitem[KS98]{KaratzasBrownianMotionStochastic1998}
Ioannis Karatzas and Steven Shreve.
\newblock {\em Brownian {{Motion}} and {{Stochastic Calculus}}}.
\newblock Graduate {{Texts}} in {{Mathematics}}. {Springer-Verlag}, {New York},
  second edition, 1998.

\bibitem[Lui22]{LuirardLoilimitemodeles2022}
Emeline Luirard.
\newblock {\em Loi Limite de Mod\`eles Cin\'etiques Inhomog\`enes En Temps}.
\newblock PhD thesis, 2022.

\end{thebibliography}

\end{document}